\documentclass[a4paper,11pt]{amsart}

\usepackage[T1]{fontenc}

\usepackage{microtype}

\usepackage{lmodern}

\usepackage{amssymb}

\numberwithin{equation}{section}

\usepackage{hyperref}

\usepackage[abbrev]{amsrefs}

\usepackage{enumerate}

\newtheorem{theorem}{Theorem}

\theoremstyle{definition}
\newtheorem{definition}{Definition}[section]
\newtheorem{openproblem}{Open Problem}[section]
\theoremstyle{remark}

\newcommand{\abs}[1]{\lvert#1\rvert}

\newcommand{\Bigabs}[1]{\Bigl\lvert#1\Bigr\rvert}
\newcommand{\dualprod}[2]{\langle #1 , #2 \rangle}
\newcommand{\st}{\: :\:}
\newcommand{\norm}[1]{\lVert#1\rVert}
\newcommand{\R}{\mathbf{R}}
\newcommand{\dif}{\,\mathrm{d}}

\newcommand{\N}{\mathbf{N}}

\newcommand{\Lin}{\mathcal{L}}

\DeclareMathOperator{\supp}{supp}
\DeclareMathOperator{\Div}{div}
\DeclareMathOperator{\Curl}{curl}

\DeclareMathOperator{\rank}{rank}

\author{Jean Van Schaftingen}
\address{Universit\'e catholique de Louvain\\
Institut de Recherche en Math\'ematique et Physique (IRMP)\\
Chemin du Cyclotron 2 bte L7.01.01\\
1348 Louvain-la-Neuve\\
Belgium}
\email{Jean.VanSchaftingen@uclouvain.be}


\subjclass%
{35B65 (26D15, 35A23, 35B45, 35R03, 43A80, 46E35, 58A10)}

\title%
[Limiting Bourgain-Brezis estimates for systems]%
{Limiting Bourgain-Brezis estimates for systems: theme and variations}

\begin{document}

\begin{abstract}
J.\thinspace{}Bourgain and H.\thinspace{}Brezis have obtained in 2002 some new and surprising estimates for systems of linear differential equations, dealing with the endpoint case \(\mathrm{L}^1\) of singular integral estimates and the critical Sobolev space \(\mathrm{W}^{1, n} (\R^n)\).
This paper presents an overview of the results, further developments over the last ten years and challenging open problems.
\end{abstract}

\maketitle

\section{Theme}

\subsection{Limiting Hodge theory for Sobolev forms}
The study of limiting estimates for systems starts from the following problem: given a function  \(g \in \mathrm{L}^n (\R^n; \R^n)\) find the best regularity that a vector field \(u : \R^n \to \R^n\) such that 
\begin{equation}
\label{eqDivergence}
  \Div u = g \qquad \text{in \(\R^n\)}
\end{equation}
can have.

If \(n \ge 2\), the equation \eqref{eqDivergence} is strongly \emph{underdetermined}. 
The standard way of finding a solution \(u\) consists in lifting the undeterminacy by solving the system
\begin{equation}
\label{eqDivCurl}
\left\{
\begin{aligned}
  \Div u &= g \qquad \text{in \(\R^n\)},\\
  \Curl u & = 0 \qquad \text{in \(\R^n\)},
\end{aligned}
\right.
\end{equation}
where \(\Curl u = Du - (Du)^*\). 
By the classical Calder\'on--Zygmund theory of singular integrals \cite{CalderonZygmund1952} (see also \cite{Stein1970}), there exists a function \(v \in \mathrm{W}^{2, n}_\mathrm{loc} (\R^n)\) that satisfies
\[
 -\Delta v = g \qquad \text{in \(\R^n\)}
\]
and 
\[
 \norm{D^2 v}_{\mathrm{L}^n} \le C_n \norm{g}_{\mathrm{L}^n}.
\]
In particular, the vector-field \(u = \nabla v\) solves the problem \eqref{eqDivCurl}, and thus also the original problem \eqref{eqDivergence}, and \(u\) satisfies the estimates
\[
  \norm{D u}_{\mathrm{L}^n} \le C_n \norm{g}_{\mathrm{L}^n}.
\]
In general, vector fields in the Sobolev space \(W^{1, n} (\R^n; \R^n)\) need not be bounded functions (see for example \citelist{\cite{Brezis2011}*{remark 9.16}\cite{AdamsFournier2003}*{remark 4.43}}). This is also not the case for our solution \(u\): L. Nirenberg has given as a counterexample the data \(g = -\Delta v\) with \(v (x) = x_1 (\log \abs{x})^\alpha\zeta (x)\), where \(\zeta\) is a suitable cut-off function and \(\alpha \in (0, \frac{n}{n - 1})\) \cite{BourgainBrezis2003}*{remark 7} (see also \cite{AdamsFournier2003}*{example 4.44}).

As this solution of the underdetermined system \eqref{eqDivCurl} is merely one out of infinitely many, we can still hope that \eqref{eqDivergence} has \emph{another solution that is bounded}.
J.\thinspace{}Bourgain and H.\thinspace{}Brezis have constructed such solutions  \citelist{\cite{BourgainBrezis2002}*{proposition 1}\cite{BourgainBrezis2003}*{proposition 1}\cite{BourgainBrezis2004}*{theorem 4}\cite{BourgainBrezis2007}*{theorem 5}}.

\begin{theorem}
\label{theoremWeakHodge}
Let \(\ell \in \{1, \dotsc,  n - 1\}\). If \(g \in \mathrm{L}^n (\R^n; \bigwedge^{\ell + 1} \R^n)\) and \(dg = 0\) in the sense of distributions, then there exists \(u \in \mathrm{L}^\infty (\R^n; \bigwedge^\ell \R^n)\), such that 
\(du = g\) in the sense of distributions.
Moreover,
\[
  \norm{u}_{\mathrm{L}^{\infty}}  \le C \norm{g}_{\mathrm{L}^{n}}.
\]
\end{theorem}

The first part of the statement can be written, in view of the classical Hodge theory \citelist{\cite{Iwaniec1995}\cite{Schwarz1995}}
\[
  d \bigl(\dot{\mathrm{W}}^{1, n} (\R^n; \textstyle{\bigwedge^\ell} \R^n)\bigr)
  \subset d \bigl(\mathrm{L}^\infty (\R^n; \textstyle{\bigwedge^\ell} \R^n)\bigr),
\]
where \(\dot{\mathrm{W}}^{1, n} (\R^n; \textstyle{\bigwedge^\ell} \R^n)\) denotes the homogeneous Sobolev space of weakly differentiable differential forms such that \(\abs{Du} \in L^n (\R^n)\).

In the theory of lifting of fractional Sobolev maps into the unit circle  \cite{BourgainBrezisMironescu2000},
theorem~\ref{theoremWeakHodge} has allowed to derive some local bound on the norm of the phase \(\norm{\varphi}_{\mathrm{L}^{n/(n - 1)}}\) in terms of \(\norm{e^{i \varphi}}_{H^{1/2}}\) \cite{BourgainBrezis2003}*{corollary 1}.
Theorem~\ref{theoremWeakHodge} was also used to reformulate a smallness assumption on a magnetic vector potential in \(\mathrm{L}^\infty (\R^n)\) as an assumption on the magnetic field in \(\mathrm{L}^n (\R^n)\) \cite{AbatangeloTerracini2011}*{p. 159}.

In comparison with the standard Hodge theory \citelist{\cite{Iwaniec1995}\cite{Schwarz1995}}, theorem~\ref{theoremWeakHodge} does not give any integrability information on the derivative \(D u\). 
J.\thinspace{}Bourgain and H.\thinspace{}Brezis have constructed a solution that satisfies \emph{both} the estimates of theorem~\ref{theoremWeakHodge} \emph{and} the classical estimates \citelist{\cite{BourgainBrezis2002}*{theorem 1}\cite{BourgainBrezis2003}*{theorem 1}\cite{BourgainBrezis2004}*{theorem 4}\cite{BourgainBrezis2007}*{theorem 5}}.

\begin{theorem}
\label{theoremStrongHodge}
Let \(\ell \in \{1, \dotsc,  n - 1\}\). If \(g \in \mathrm{L}^n (\R^n; \bigwedge^{\ell + 1} \R^n)\) and \(dg = 0\) in the sense of distributions, then there exists \(u \in \mathrm{L}^\infty (\R^n; \bigwedge^\ell \R^n)\), such that
\(du = g\) in the sense of distributions, \(u\) is continuous and \(Du \in \mathrm{L}^n (\R^n)\).
Moreover,
\[
  \norm{u}_{\mathrm{L}^{\infty}} + \norm{Du}_{\mathrm{L}^{n}} \le C \norm{g}_{\mathrm{L}^{n}}.
\]  
\end{theorem}

The first part of the statement can be written, in view of the classical Hodge theory \citelist{\cite{Iwaniec1995}\cite{Schwarz1995}} as
\[
  d \bigl(\dot{\mathrm{W}}^{1, n} (\R^n; \textstyle{\bigwedge^\ell} \R^n)\bigr)
  = d \bigl(\dot{\mathrm{W}}^{1, n} (\R^n; \textstyle{\bigwedge^\ell} \R^n) \cap \mathrm{L}^\infty (\R^n; \textstyle{\bigwedge^\ell} \R^n)\bigr),
\]
or as 
\begin{multline*}
  \dot{\mathrm{W}}^{1, n} (\R^n; \textstyle{\bigwedge^\ell} \R^n)
  = \dot{\mathrm{W}}^{1, n} (\R^n; \textstyle{\bigwedge^\ell} \R^n) \cap \mathrm{L}^\infty (\R^n; \textstyle{\bigwedge^\ell} \R^n)\\
  + d \bigl(\dot{\mathrm{W}}^{2, n} (\R^n; \textstyle \bigwedge^{\ell - 1} \R^n)\bigr),
\end{multline*}
that is, every \(\dot{\mathrm{W}}^{1, n}\)--Sobolev \(\ell\)--form is bounded up to an exact form.

When \(\ell = n - 1\), theorem~\ref{theoremStrongHodge} states that every \(g \in W^{1, n} (\R^n; \R^n)\) can be written as 
\begin{equation}
\label{eqRieszDecomposition}
  g = \mathcal{R} v,
\end{equation}
with \(v \in W^{1, n} (\R^n; \R^n) \cap L^\infty (\R^n; \R^n)\) and the vector Riesz transform is defined by its Fourier transform \(\widehat{\mathcal{R} v} (\xi) = i \xi \cdot v (\xi)/\abs{\xi}\). As noted by J. Bourgain and H. Brezis \cite{BourgainBrezis2003}, the decomposition \eqref{eqRieszDecomposition} is a refined version of the Fefferman-Stein decomposition \citelist{\cite{FeffermanStein1972}*{theorem 3}\cite{Uchiyama1982}} which states that  \(g \in \mathrm{BMO} (\R^n)\) if and only if it can be decomposed as 
\begin{equation*}
  g = w + \mathcal{R} v,
\end{equation*}
with \(w \in L^\infty (\R^n)\) and \(v \in L^\infty (\R^n; \R^n)\).

Theorem~\ref{theoremStrongHodge} has allowed to obtain uniform \(\mathrm{L}^{n/(n - 1)}\) estimates on the gradient of minimizers of the Ginzburg--Landau functional \citelist{\cite{BourgainBrezis2004}*{theorem 5}\cite{BourgainBrezis2007}*{theorem 21}} (see also \citelist{\cite{BourgainBrezisMironescu2004}*{theorem 11}\cite{BethuelOrlandiSmets2004}*{proposition 5.1}}).

\subsection{Bourgain--Brezis linear estimates}

The existence theorem~\ref{theoremWeakHodge} can be reformulated as a linear estimate \citelist{\cite{BourgainBrezis2004}\cite{BourgainBrezis2007}*{theorem 1\cprime}\cite{VanSchaftingen2004Divf}*{corollary 1.4}}.

\begin{theorem}
\label{theoremWeakEstimate}
Let \(\ell \in \{1, \dotsc, n - 1\}\). There exists \(C > 0\) such that for every \(f \in C^\infty_c (\R^n; \bigwedge^{\ell} \R^n)\) and every \(\varphi \in C^\infty_c (\R^n; \bigwedge^{n - \ell} \R^n)\), if \(df = 0\), then 
\[
  \Bigabs{\int_{\R^n} f \wedge \varphi}
  \le C \norm{f}_{\mathrm{L}^1} \norm{d \varphi}_{\mathrm{L}^n}.
\]
\end{theorem}

Theorem~\ref{theoremWeakEstimate} would be a consequence of a critical Sobolev embedding of \(\dot{\mathrm{W}}^{1, n} (\R^n)\) in \(\mathrm{L}^\infty (\R^n)\) which is well-known to fail (see for example \cite{Brezis2011}*{remark 9.16}). 
The estimate is on the integral of the form \(f \wedge \varphi\) and not of the density \(\abs{f \wedge \varphi}\); it results from a compensation phenomenon which is reminiscent of div-curl estimates \citelist{\cite{CoifmanLionsMeyerSemmes1993}\cite{Murat1978}\cite{Tartar1978}\cite{Tartar1979a}\cite{Tartar1979b}}.

When \(k = 1\), theorem~\ref{theoremWeakEstimate} is equivalent with the classical Gagliardo--Nirenberg--Sobolev estimate \citelist{\cite{Gagliardo1958}\cite{Nirenberg1959}*{p.\thinspace 125}} (see also \citelist{\cite{Brezis2011}*{theorem 9.9}\cite{AdamsFournier2003}*{theorem 4.31}\cite{Stein1970}*{V.2.5}\cite{Mazya2011}*{(1.4.14)}})
\begin{equation}
\label{inequalityGagliardoNirenbergSobolev}
 \norm{u}_{\mathrm{L}^{n / (n - 1)}} \le C \norm{D u}_{\mathrm{L}^1}.
\end{equation}

We explain how theorems~\ref{theoremWeakHodge} and \ref{theoremWeakEstimate} are equivalent \citelist{\cite{BourgainBrezis2007}*{remark 8}\cite{BourgainBrezis2004}} (see also \cite{VanSchaftingen2004Divf}).
First by theorem~\ref{theoremWeakHodge}, for every \(\varphi \in C^\infty_c (\R^n; \bigwedge^{n - \ell} \R^n)\), there exists \(u \in \mathrm{L}^\infty (\R^n; \bigwedge^{n - \ell} \R^n)\) such that \(du = d \varphi\) and \(\norm{u}_{\mathrm{L}^{\infty}}  \le C \norm{d \varphi}_{\mathrm{L}^{n}}\). Moreover, by the classical Calder\'on--Zygmund elliptic regularity estimates, there exists \(\zeta \in \dot{\mathrm{W}}^{2, n} (\R^n;\bigwedge^{n - \ell - 1} \R^n)\) such that 
\[
\left\{
\begin{aligned}
  d \zeta & = \varphi - u && \text{in \(\R^n\)},\\
  d^* \zeta & = 0 && \text{in \(\R^n\)},
\end{aligned}
\right.
\]
(\(d^* \zeta\) denotes the exterior codifferential of the differential form \(\zeta\)).
Hence,
\[
  \Bigabs{\int_{\R^n} f \wedge \varphi} =  
  \Bigabs{\int_{\R^n} f \wedge (u + d \zeta)}
  = \Bigabs{\int_{\R^n} f \wedge u}
  \le C \norm{f}_{\mathrm{L}^{1}}\norm{d \varphi}_{\mathrm{L}^{n}}.
\]

Conversely, if \(g \in \mathrm{L}^n (\R^n; \bigwedge^{\ell + 1} \R^n)\), and \(d g = 0\), by the classical Hodge theory in Sobolev spaces, there exists \(v \in \dot{\mathrm{W}}^{1, n} (\R^n; \bigwedge^{\ell}\R^n)\) such that 
\[
\left\{
\begin{aligned}
  d v & = g &&\text{in \(\R^n\)},\\
  d^* v & = 0 &&\text{in \(\R^n\)}.
\end{aligned}
\right.
\]
For every \(\psi \in C^\infty_c (\R^n; \bigwedge^{n - \ell - 1}\R^n)\), by theorem~\ref{theoremWeakHodge}
\[
  \Bigabs{\int_{\R^n} d \psi \wedge v} \le C \norm{d \psi}_{\mathrm{L}^{1}}\norm{D v}_{\mathrm{L}^{n}}
  \le C' \norm{d \psi}_{\mathrm{L}^{1}} \norm{g}_{\mathrm{L}^{n}}.
\]
By the classical Hahn-Banach theorem (see for example \cite{Brezis2011}*{corollary 1.2}) and the representation of linear functionals on \(\mathrm{L}^1\) (see for example \cite{Brezis2011}*{theorem 4.14}), there exists \(u \in \mathrm{L}^\infty (\R^n; \bigwedge^{\ell} \R^n)\) such that for every \(\psi \in C^\infty_c (\R^n; \bigwedge^{n - \ell - 1}\R^n)\),
\[
  \int_{\R^n} d \psi \wedge v = \int_{\R^n} d \psi \wedge u.
\]
By construction of \(u\), we conclude that for every
\[
  \int_{\R^n} d \psi \wedge u = \int_{\R^n} d \psi \wedge v = (-1)^{n - \ell} \int_{\R^n} \psi \wedge g,
\]
that is, \(d u = g\) in the sense of distributions.

Theorem~\ref{theoremWeakEstimate} was used to show that if \(g \in \mathrm{W}^{1, p} (\Omega, \R^n) \cap \mathrm{L}^q (\Omega; \R^n)\), \(\frac{1}{q} + \frac{n - 1}{p} = 1\) and if \(\det (Dg) = \Div \mu\), then the measure \(\abs{\mu}\) does not charge sets of null \(\mathrm{W}^{1, n}\)--capacity \cite{BrezisNguyen2011}*{proposition 2}. Theorem~\ref{theoremWeakEstimate} yields a representation of divergence-free measures in the study of limiting div-curl lemmas \citelist{\cite{BrianeCasadoDiazMurat2009}*{theorem 3.1}}.
Theorem~\ref{theoremWeakEstimate} also allows to obtain endpoint Strichartz estimate for the linear wave and Schr\"odinger equations with space divergence-free data \cite{ChanilloYung2012}.

As a consequence of theorem~\ref{theoremWeakHodge} and the classical elliptic regularity theory, we have a Gagliardo--Nirenberg--Sobolev inequality for forms \cite{BourgainBrezis2007}*{corollary 17} (see also \cite{LanzaniStein2005}): if \(\ell \in \{2, \dotsc, n - 2\}\), then
\begin{equation}
\label{ineqHodgeSobolev}
 \norm{u}_{\mathrm{L}^{n / (n - 1)}}\le C \bigl(\norm{du}_{\mathrm{L}^{1}}+ \norm{d^* u}_{\mathrm{L}^{1}}\bigr);
\end{equation}
the inequality still holds for \(\ell \in \{1, n - 1\}\) provided \(d^*u = 0\) if \(\ell = 1\) and \(du\) if \(\ell = n - 1\). In particular, there is no such estimate for \(n = 2\). 
The vanishing of \(du\) or \(d^*u\) can be replaced by an estimate in the real Hardy space \(\mathcal{H}^1 (\R^n)\) \cite{LanzaniStein2005}.
The inequality \eqref{ineqHodgeSobolev} was used in the Chern--Weil theory for Sobolev connections on Sobolev bundles \cite{Isobe2010}*{proposition 4.1}.
This family of inequalities can be extended to higher-order analogues of the exterior derivative \cite{Lanzani2013}.

The inequality \eqref{ineqHodgeSobolev} would be a consequence of the classical Gagliardo--Nirenberg--Sobolev inequality \eqref{inequalityGagliardoNirenbergSobolev} and of the Gaffney inequality
\[
  \norm{Du}_{\mathrm{L}^{1}}\le C \bigl(\norm{du}_{\mathrm{L}^{1}}+ \norm{d^* u}_{\mathrm{L}^{1}}\bigr);
\]
the latter inequality does not hold \citelist{\cite{Ornstein1962}\cite{BourgainBrezis2007}*{remark 1}} (see also \citelist{\cite{ContiFaracoMaggi2005}\cite{KirchheimKristensen2011}\cite{CKPaper}}).

Theorem~\ref{theoremWeakHodge} also allows to obtain estimates when the classical Calder\'on--Zygmund theory fails: for every \(u \in C^\infty_c (\R^n; \R^n)\), if \(\Div u = 0\), then \citelist{\cite{BourgainBrezis2004}*{corollary 1 and remark 5}\cite{BourgainBrezis2007}*{theorem 2}}
\begin{equation}
\label{ineqLaplaceSobolev}
 \norm{D u}_{\mathrm{L}^{n / (n - 1)}}\le C \norm{\Delta u}_{\mathrm{L}^{1}}.
\end{equation}
Without the divergence-free condition, this estimate fails when \(n > 1\), as can be seen by taking \(u\) to be an approximation of the Green function of Laplacian on \(\R^n\). 
Even under divergence-free condition, the inequality
\[
 \norm{D^2 u}_{\mathrm{L}^{1}}\le C \norm{\Delta u}_{\mathrm{L}^{1}}
\]
does not hold \citelist{\cite{Ornstein1962}\cite{BourgainBrezis2007}*{remark 1}} (see also \citelist{\cite{ContiFaracoMaggi2005}\cite{KirchheimKristensen2011}\cite{CKPaper}}).
The estimate \eqref{ineqLaplaceSobolev} would be a consequence of the latter inequality combined with the Gagliardo--Nirenberg--Sobolev inequality \ref{inequalityGagliardoNirenbergSobolev}.

If \(\dot{\mathrm{W}}^{-k, p} (\R^n)\) is the set of distributions which are \(k\)-th derivatives of \(\mathrm{L}^p\) functions for \(k \in \N\) and \(p \in (1, \infty)\), that is, the set of distributions \(f\) such that 
\[
 \norm{f}_{\dot{\mathrm{W}}^{-k, p}} = \sup \{ \dualprod{f}{\varphi} : \varphi \in C^\infty_c (\R^n) \text{ and }\norm{D \varphi}_{\mathrm{L}^{p / (p - 1)}} \le 1 \} < \infty.
\]
the estimate of theorem~\ref{theoremWeakEstimate} can be rewritten as
\[
 \norm{f}_{\dot{\mathrm{W}}^{-1,n/(n - 1)}}
 \le C \norm{f}_{\mathrm{L}^1}.
\]

It is known that when \(n \ge 2\), \(\mathrm{L}^1 (\R^n) \not \subset \dot{\mathrm{W}}^{-1,n/(n - 1)} (\R^n)\).
J.\thinspace{}Bourgain and H.\thinspace{}Brezis have characterized by their divergence the vector fields \(f \in \mathrm{L}^1(\R^n; \R^n)\) that are in \(\mathrm{W}^{-1,n/(n - 1)} (\R^n;\R^n)\) \cite{BourgainBrezis2004}*{theorem 4\cprime}.

\begin{theorem}
\label{theoremStrongEstimate}
Let \(\ell \in \{2, \dotsc, n - 1\}\).
If \(f \in \mathrm{L}^1 (\R^n; \textstyle\bigwedge^{\ell} \R^n)\), then 
\[f \in \dot{\mathrm{W}}^{-1, n/(n - 1)} (\R^n; \textstyle\bigwedge^{\ell} \R^n)\] if and only if \[d f \in \dot{\mathrm{W}}^{-2, n/(n - 1)} (\R^n; \textstyle\bigwedge^{\ell+1} \R^n).\]
Moreover
\[
 \norm{f}_{\dot{\mathrm{W}}^{-1, n / (n - 1)}}
 \le C\bigl(\norm{f}_{\mathrm{L}^1} + \norm{df}_{\dot{\mathrm{W}}^{-2, n / (n - 1)}}\bigr).
\]
\end{theorem}

Theorem~\ref{theoremStrongEstimate} is equivalent to theorem~\ref{theoremStrongHodge} in the same way that theorem~\ref{theoremWeakEstimate} is equivalent to theorem~\ref{theoremWeakHodge}.

Theorem~\ref{theoremStrongEstimate} was used to obtain a generalized Korn type inequality in the derivation of a strain gradient theory for plasticity by homogenization of dislocations \cite{GarroniLeoniPonsiglione2010}.

\subsection{Estimates for circulation integrals}
Theorems~\ref{theoremWeakHodge} and \ref{theoremWeakEstimate} are equivalent to the following geometrical inequality of J.\thinspace{}Bourgain, H.\thinspace{}Brezis and P. Mironescu \cite{BourgainBrezisMironescu2004}*{proposition 4}.

\begin{theorem}
\label{theoremCirculation}
If \(\Gamma \subset \R^n\) is a closed rectifiable curve of length \(\abs{\Gamma}\) and \(\varphi \in C^\infty_c (\R^n; \bigwedge^1 \R^n)\), then 
\[
 \Bigabs{\int_{\Gamma} \varphi} \le C \abs{\Gamma} \norm{D \varphi}_{\mathrm{L}^n}.
\]
\end{theorem}

Here \(\int_{\Gamma} \varphi\) denotes the \emph{circulation integral} of the form \(\varphi\) along the curve \(\Gamma\).

Again this estimate would be a consequence of the failing critical Sobolev embedding of \(\mathrm{W}^{1, n} (\R^n)\) into \(\mathrm{L}^\infty (\R^n)\); it is a consequence of some compensation phenomenon that appears since the curve \(\Gamma\) is \emph{closed}.
When \(n = 2\), theorem~\ref{theoremCirculation} is a direct consequence of the Green--Stokes integration formula and of the classical isoperimetric inequality.

Theorem~\ref{theoremCirculation} can be deduced from theorem~\ref{theoremWeakEstimate} by applying the estimate to regularizations by convolution the divergence-free vector measure \(t\mathcal{H}^1_{\vert \Gamma}\), where \(\mathcal{H}^1_{\vert \Gamma}\) is the one-dimensional Hausdorff measure restricted to \(\Gamma\) and \(t\) is the unit tangent vector to \(\Gamma\).
Conversely, S. Smirnov has showed that any divergence-free measure is the limit of convex combinations of measures of the form \(t\mathcal{H}^1_{\vert \Gamma}\), with a suitable control on the norms \cite{Smirnov1994}; this allows to deduce theorem~\ref{theoremWeakEstimate} for \(\ell = n - 1\) from theorem~\ref{theoremCirculation} \cite{BourgainBrezis2004}; the cases \(\ell < n - 1\) follow immediately. This arguments shows that the constant in theorem~\ref{theoremWeakHodge} with \(\ell = n - 1\) and theorem~\ref{theoremCirculation} can be taken to be the same.

Theorem~\ref{theoremCirculation} has been used to obtain \(\mathrm{L}^{n/(n - 1)}\) bounds on minimizers of the Ginzburg--Landau equation \cite{BethuelOrlandiSmets2004}*{proposition 5.1}.

The geometrical nature of the estimate of theorem~\ref{theoremCirculation} has raised the problem of the value of optimal constants and whether they are achieved \cite{BrezisVanSchaftingen2008}.

Theorem~\ref{theoremCirculation} generalizes to surfaces \cite{VanSchaftingen2004BBM}: if \(\Sigma\) is an \(\ell\)--dimensional oriented surface and \(\varphi \in C^\infty_c (\R^n; \bigwedge^\ell \R^n)\), then 
\[
 \Bigabs{\int_{\Sigma} \varphi} \le C \mathcal{H}^\ell (\Sigma) \norm{d \varphi}_{\mathrm{L}^n}.
\]

\section{About the proofs}

In this section we explain the proofs of the results presented above. First, J.\thinspace{}Bourgain and H.\thinspace{}Brezis have observed that the construction of the solution \(u\) in theorem~\ref{theoremWeakHodge} --- and a fortiori in the stronger theorem~\ref{theoremStrongHodge} --- \emph{cannot be linear} \cite{BourgainBrezis2003}*{proposition 2}.
\begin{theorem}
\label{theoremNonlinear}
Let \(\ell \in \{1, \dotsc, n - 1\}\). There does not exist a linear operator \(K : \mathrm{L}^n (\R^n; \bigwedge^{\ell + 1} \R^n) \to \mathrm{L}^\infty (\R^n; \bigwedge^\ell \R^n)\) such that for every \(f \in \mathrm{L}^n (\R^n; \bigwedge^{\ell + 1} \R^n)\), \(d(K(f)) = f\).
\end{theorem}

As the deduction of theorem~\ref{theoremWeakHodge} from theorem~\ref{theoremWeakEstimate} above is based on the nonconstructive Hahn-Banach theorem on \(\mathrm{L}^1\), the corresponding map does not need be linear.

Theorem~\ref{theoremNonlinear} has a harmonic analysis proof and a geometric functional analysis proof \cite{BourgainBrezis2003}. The harmonic analysis proof begins by asssuming, by an averaging argument, that \(K\) is a convolution operator and derives then a contradiction. The geometric functional analysis proof consists in noting that \(K^* \circ d^*\) would be a factorization of the identity map from \(\mathrm{W}^{1,1}\) to \(\mathrm{L}^{n / (n - 1)}\) through \(\mathrm{L}^1\) and that such factorization is impossible by Grothendieck's theorem on absolutely summing operators \citelist{\cite{Grothendieck1953}\cite{Wojtasczyk1991}*{theorem III.F.7}}.

The main analytical tool in the proof of theorem~\ref{theoremStrongHodge} is an approximation lemma~for functions in \(\mathrm{W}^{1, n} (\R^n)\) \citelist{\cite{BourgainBrezis2004}*{theorem 6}\cite{BourgainBrezis2007}*{theorem 11 and (5.25)}} (see also \cite{BourgainBrezis2003}*{(5.2) and (5.3), (6.22)}).

\begin{theorem}
\label{theoremApproximation}
Let \(T \in \Lin (\R^n;\R^p)\). If \(\ker T \ne \{0\}\), then for every \(\varepsilon > 0\), there exists \(C_\varepsilon\) such that for every \(v \in \mathrm{W}^{1, n} (\R^n)\) there exists \(u \in C^\infty (\R^n)\) that satisfies
\begin{gather*}
  \norm{T (\nabla u - \nabla v)}_{\mathrm{L}^n} \le \varepsilon \norm{D v}_{\mathrm{L}^n},\\
  \norm{\nabla u}_{\mathrm{L}^n} + \norm{u}_{\mathrm{L}^\infty} \le C_\varepsilon \norm{D v}_{\mathrm{L}^n}.
\end{gather*}
\end{theorem}

The proof of theorem~\ref{theoremApproximation} is constructive and based on a Littlewood--Paley decomposition \cite{BourgainBrezis2007}.
This approximation result has been extended to some subscale of Triebel--Lizorkin spaces when \(\rank T = 1\) \cite{BousquetMironescuRuss2013}*{proposition 3.1} and to classical Sobolev spaces in the noncommutative setting of homogeneous groups \cite{WangYung}*{lemma~1.7}.

It would be interesting to find a simpler proof of theorem~\ref{theoremApproximation}.

We will now state a theorem that provides bounded solutions to overdetermined systems \cite{BourgainBrezis2007}*{theorem 10 and 10\cprime} reformulated in the spirit of more recent works \cite{VanSchaftingen2013}. We introduce therefore the notion of adcanceling opetarors.

\begin{definition}
\label{definitionAdcanceling}
A homogeneous differential operator \(T (D)\) from \(E\) to \(V\) is \emph{adcanceling} if 
\[
  \bigcap_{\xi \in \R^n \setminus \{0\}} T (\xi)^*[V] = \{0\}.
\]
\end{definition}

By a classical linear algebra argument, \(T (D)\) is adcanceling if and only if
\[
  \operatorname{span} \Bigl( \bigcup_{\xi \ne 0} \ker T (\xi) \Bigr) = E
\]
that is, there exists a basis \(e_1, \dotsc, e_k\) of \(E\) and vectors \(\xi_1, \dotsc, \xi_k\) in \(\R^n\) such that for every \(i \in \{1, \dotsc, k\}\), \(L (\xi_i) [e_i] = 0\) \cite{VanSchaftingen2013}*{\S 6.2}.

We now state the theorem of J. Bourgain and H. Brezis that provides bounded solutions to overdetermined systems \cite{BourgainBrezis2007}*{theorem 10 and 10\cprime}.

\begin{theorem}
\label{theoremMainClosed}
Let \(Y\) be a Banach space and let \(S : \mathrm{W}^{1, n} (\R^n; E) \to Y\) be a bounded linear operator.
If \(S\) has closed range and if there exists an adcanceling first-order homogeneous differential operator \(T (D)\) such that for every \(\varphi \in C^\infty_c (\R^n; E) \) 
\[
  \norm{S (\varphi)}_Y \le \norm{T (D) \varphi}_{\mathrm{L}^n},
\]
then for every \(v \in \dot{\mathrm{W}}^{1, n} (\R^n; E)\) there exists \(u \in \dot{\mathrm{W}}^{1, n} (\R^n; E) \cap \mathrm{L}^\infty (\R^n; E)\) such that 
\(
  S u = S v 
\)
and 
\[
  \norm{u}_{\mathrm{L}^\infty} + \norm{D u}_{\mathrm{L}^n} \le C \norm{D v}_{\mathrm{L}^n}.
\]
\end{theorem}

Let us first see how theorem~\ref{theoremStrongHodge} follows from theorem~\ref{theoremMainClosed}.

\begin{proof}[Proof of theorem~\ref{theoremStrongHodge} \cite{BourgainBrezis2007}*{proof of theorem 5}]
If \(df = 0\) and \(f \in \bigwedge^{\ell + 1} \R^n\), there exists \(v \in \dot{\mathrm{W}}^{1, n} (\R^n; \bigwedge^{\ell} \R^n)\) such that \(dv = f\).
We are now going to apply theorem~\ref{theoremMainClosed}.
We take \(E = \bigwedge^{\ell} \R^n\) and \(V = \bigwedge^{\ell + 1} \R\) and we define \(S = T (D) = d\). We observe that 
\[
  S \bigl(\mathrm{W}^{1, n} (\R^n; \textstyle{\bigwedge^{\ell}}\R^n)\bigr) = 
  \bigl\{ f \in \mathrm{L}^{n} (\R^n; \textstyle{\bigwedge^{\ell}}\R^n)\;:\; df = 0\bigr\}
\]
is closed in \(Y = \mathrm{L}^{n} (\R^n; \textstyle{\bigwedge^{\ell}}\R^n)\) and that
\begin{multline*}
 \operatorname{span} \Bigl( \bigcup_{\xi \ne 0} \ker T (\xi) \Bigr) \\= 
 \operatorname{span} \bigl\{\alpha \in \textstyle\bigwedge^{\ell} \R^n \st \text{there exist \(\xi \in \R^n\) such that \(\xi \wedge \alpha = 0\)}\bigr\} = E;
\end{multline*}
the latter equality holds since \(\ell \ge 1\).
\end{proof}

Theorem~\ref{theoremStrongHodge} can be used to prove theorem~\ref{theoremStrongEstimate} in the spirit of our proof of theorem~\ref{theoremWeakEstimate} from theorem~\ref{theoremWeakHodge} in the previous section.

\medbreak

We now explain the proof of theorem~\ref{theoremMainClosed} from theorem~\ref{theoremApproximation}.

\begin{proof}[Proof of theorem~\ref{theoremMainClosed} \cite{BourgainBrezis2007}*{proof of theorem 11}]
Since \(S\) has closed range, by the open mapping theorem (see for example \cite{Brezis2011}*{theorem 2.6}, there exists \(w \in \dot{\mathrm{W}}^{1, n} (\R^n; E)\) such that 
\(S v = S w\) and 
\[
  \norm{D w}_{\mathrm{L}^n} \le C \norm{S v}_{Y}.
\]
By theorem~\ref{theoremApproximation} and by definition~\ref{definitionAdcanceling}, for every \(\varepsilon > 0\) there exists \(C_\varepsilon > 0\) such that for every \(w \in \dot{\mathrm{W}}^{1, n} (\R^n; E)\), there exists \(u \in C^\infty (\R^n; E)\) satisfying
\begin{gather}
  \norm{T (D) (w - u)}_{\mathrm{L}^n} \le \varepsilon \norm{D w}_{\mathrm{L}^n},\\
  \norm{Dw}_{\mathrm{L}^n} + \norm{w}_{\mathrm{L}^\infty} \le C_\varepsilon \norm{D w}_{\mathrm{L}^n}
\end{gather}
In particular, by our assumption on \(T (D)\),
\[
  \norm{S v  - S u}_Y=\norm{S (w - u)}_Y \le \norm{T (D) (w - u)}_{\mathrm{L}^n}
  \le \varepsilon \norm{D w}_{\mathrm{L}^n} \le C \varepsilon \norm{Sv}_Y.
\]
If we now choose \(\varepsilon = \frac{1}{2C}\), we have 
\begin{gather*}
  \norm{S v - S u}_Y \le \frac{1}{2} \norm{S v}_Y,\\
  \norm{D u}_{\mathrm{L}^n} + \norm{u}_{\mathrm{L}^\infty} \le \frac{C_\varepsilon}{2C} \norm{S v}_Y.
\end{gather*}
By an iterative argument as in the classical proof of the open mapping theorem, see for example \cite{Brezis2011}*{proof of theorem 2.6}, we can thus construct \(u_i \in \dot{\mathrm{W}}^{1, n} (\R^n; E) \cap \mathrm{L}^\infty (\R^n; E)\) such that 
\begin{gather*}
  \norm{S v - S u_{i + 1}}_Y \le \frac{1}{2} \norm{S v - S u_i}_Y,\\
  \norm{D u_{i + 1} - D u_i}_{\mathrm{L}^n} + \norm{u_{i + 1} - u_i}_{\mathrm{L}^\infty} \le \frac{C_\varepsilon}{2C} \norm{S v - S u_i}_Y;
\end{gather*}
this sequences converges to the desired solution.
\end{proof}

This solutions constructed by this iterative argument in the spirit of the classical proof of the closed graph theorem have been studied as hierarchical solutions \cite{Tadmor}.

The strategy of proof outlined here above relies essentially on theorem~\ref{theoremApproximation}, which does not have yet an elementary proof.
However, the weaker theorem~\ref{theoremWeakEstimate} has a short proof \cite{VanSchaftingen2004Divf} (see also \citelist{\cite{LanzaniStein2005}*{proof of lemma~1}\cite{MitreaMitrea2009}*{proof of proposition 2}}). 

\begin{proof}[Direct proof of theorem~\ref{theoremWeakEstimate}]
Without loss of generality, we assume that \(\ell = n - 1\) and that \(\varphi (x)= \varphi_n (x) dx_n\). For every \(t \in \R\), if we define 
\(\varphi_n^t (y, z) = \varphi_n (y, t)\), we have for every \(t \in \R\) the immediate bound
\[
  \Bigabs{\int_{\R^{n - 1} \times \{t\}} f \wedge \varphi_n}
  \le \Bigl(\int_{\R^{n - 1}} \abs{f}\Bigr) \norm{\varphi_n^t}_{\mathrm{L}^\infty}.
\]
On the other hand, by the Stokes--Cartan formula, since \(df = 0\),
\[
\begin{split}
  \Bigabs{\int_{\R^{n - 1} \times \{t\}} f \wedge \varphi_n}
  &= \Bigabs{\int_{\R^{n - 1} \times (-\infty, t)} f \wedge d\varphi_n^t}\\
  &\le \Bigl(\int_{\R^{n - 1} \times (-\infty, t)} \abs{f}\Bigr) \norm{D \varphi_n^t}_{\mathrm{L}^\infty}\\
  &\le \Bigl(\int_{\R^{n}} \abs{f}\Bigr) \norm{D \varphi_n^t}_{\mathrm{L}^\infty}
\end{split}
\]
By a straightforward interpolation argument, this implies that for every \(\alpha \in (0, 1)\),
\[
  \Bigabs{\int_{\R^{n - 1} \times \{t\}} f \wedge \varphi_n}
  \le C \Bigl(\int_{\R^{n - 1} \times (-\infty, t)} \abs{f}\Bigr)^\alpha \Bigl(\int_{\R^{n - 1}} \abs{f}\Bigr)^{1 - \alpha} \abs{\varphi_n^t}_{C^{0, \alpha}},
\]
where the H\"older seminorm is defined by 
\[
  \abs{\psi}_{C^{0, \alpha}} = \sup_{x, y \in \R^{n - 1}} \frac{\abs{\psi (x) - \psi (y)}}{\abs{x - y}^\alpha}.
\]
In particular, if \(\alpha = \frac{1}{n}\), we have by the Morrey--Sobolev embedding on \(\R^{n - 1}\) (see for example \citelist{\cite{Brezis2011}*{theorem 9.12}\cite{AdamsFournier2003}*{lemma 4.28}\cite{Mazya2011}*{theorem 1.4.5 (f)}})
\[
  \Bigabs{\int_{\R^{n - 1} \times \{t\}} f \wedge \varphi_n}
  \le C' \Bigl(\int_{\R^{n}} \abs{f}\Bigr)^\frac{1}{n} \Bigl(\int_{\R^{n - 1} \times \{t\}} \abs{f}\Bigr)^{1 - \frac{1}{n}} \Bigl(\int_{\R^{n - 1} \times \{t\}} \abs{D \varphi_n}^n \Bigr)^\frac{1}{n}.
\]
The conclusion follows by H\"older's inequality.  
\end{proof}

This strategy of proof goes back to the elementary proof of the estimate on circulation integrals of theorem~\ref{theoremCirculation} \cite{VanSchaftingen2004BBM}.
The idea of working with hyperplanes and concluding with H\"older's inequality is reminiscent of the original proof of the Gagliardo--Nirenberg--Sobolev inequality \citelist{\cite{Gagliardo1958}\cite{Nirenberg1959}*{p.\thinspace 125}} (see also \cite{Brezis2011}*{theorem 9.9}). 

The two main properties of the Sobolev space \(\dot{W}^{1, n}\) that are used in this argument are a Morrey-type embedding in a space of H\"older continuous functions and a Fubini-type property. The latter property is satisfied by fractional Sobolev spaces \(W^{s, p}\) \citelist{\cite{Strichartz1967}\cite{Strichartz1968}} and by Triebel--Lizorkin spaces \(F^{s, p}_q\) \citelist{\cite{Kaljabin1980}\cite{Triebel1983}*{theorem 2.5.13}\cite{RunstSickel1996}*{theorem 2.3.4/2}} allowing to adapt the proof in that setting \citelist{\cite{BourgainBrezis2004}*{remark 1}\cite{VanSchaftingen2004Divf}*{remark~5} \cite{BourgainBrezis2007}*{remark 11}\cite{VanSchaftingen2010}*{proof of proposition 2.1}}, but is not satisfied by the Sobolev--Lorentz spaces \(W^{1, (p, q)}\) if \(q > p\) \cite{Kolyada2013} or by the Besov space \(B^{s, p}_q\) if \(q \ne p\) \cite{Triebel2001}*{theorem 4.4}. 

The proof has also been adapted by constructing \(\varphi_n^t\) more carefully than by a mere extension to estimates under higher-order conditions \citelist{\cite{VanSchaftingen2008}\cite{VanSchaftingen2004ARB}\cite{BrianeCasadoDiaz2011}*{lemma~2.4}}, and to homogeneous groups \cite{ChanilloVanSchaftingen2009}.

\section{Variations}

\subsection{Boundary estimates}
The results presented above were all concerned about the entire Euclidean space \(\R^n\).
They also have counterparts on the torus \(\mathbf{T}^n\).

On a domain with a boundary, it is not clear a priori which boundary conditions are admissible in this theory.
The problem was first settled for theorem~\ref{theoremStrongHodge} on the cube \cite{BourgainBrezis2007}*{theorems 5\cprime{} and 5\cprime{}\cprime{}} and then on a domain with a smooth boundary \cite{BrezisVanSchaftingen2007}*{lemma 4.4}.

\begin{theorem}
\label{theoremBoundaryHodge}
Let \(\Omega \subset \R^n\) be a smooth domain and let \(1 \le k \le n - 1\).
For every \(v \in \mathrm{W}^{1, n} (\Omega;\bigwedge^k \R^n)\), 
there exist \(u \in \mathrm{W}^{1, n} (\Omega; \bigwedge^{k-1} \R^n) \cap C (\Omega;\bigwedge^k\R^n)\) and \(w \in \mathrm{W}^{2,n} (\Omega;\bigwedge^{k-1}\R^n)\) such that 
\[
  u=v + d w.
\]
satisfying
\[
  \norm{u}_{\mathrm{W}^{1, n}}+\norm{u}_{\mathrm{L}^{\infty}} +\norm{w}_{\mathrm{W}^{2,n}}
\le C \norm{v}_{\mathrm{W}^{1, n}}.
\]
Moreover, if \(\mathbf{t} v=0\) on \(\partial \Omega\),
then \(u=0\) and \(w=0\) on \(\partial \Omega\) and if \(u=0\) on \(\partial \Omega\), then \(D w = 0\) on \(\partial \Omega\).
\end{theorem}

Here \(\mathbf{t} v\) denotes the tangential component on \(\partial \Omega\) of the form \(v\) \cite{Schwarz1995}*{(1.2.25)}. The inequalities on the boundary are interpreted in the sense of traces.

The presentation of theorem~\ref{theoremBoundaryHodge} differs from that of theorem~\ref{theoremStrongHodge}. 
In the case of a domain with nontrivial topology writing \(u = v + d w\) is stronger than \(du = dv\); the latter statement was however more convenient to state the weaker theorem~\ref{theoremWeakHodge}.

If \(\Omega\) is a cube, the proof of theorem~\ref{theoremBoundaryHodge} relies on the counterpart of theorem~\ref{theoremMainClosed} on a cube which is based on the counterpart of theorem~\ref{theoremApproximation} on a cube \cite{BourgainBrezis2007}*{corollary 15}.
The surjective of the trace and of normal derivative operator \cite{LionsMagenes1961} allows to obtain \(u = 0\) on \(\partial \Omega\) (see also \citelist{\cite{BourgainBrezis2003}*{proof of theorem 2}\cite{AdamsFournier2003}*{theorem 5.19}}).
A partition of the unity allows to pass to a general smooth domain \cite{BrezisVanSchaftingen2007}*{lemma~3.3}.

Similarly, the counterpart of theorem~\ref{theoremWeakEstimate} is 

\begin{theorem}
\label{theoremWeakDomain}
Let \(\Omega \subset \R^n\) be a smooth domain and let \(\ell \in \{1, \dotsc, n - 1\}\). There exists \(C > 0\) such that for every \(f \in C^\infty (\Bar{\Omega}; \bigwedge^{\ell} \R^n)\) and every \(\varphi \in C^\infty (\Bar{\Omega}; \bigwedge^{n - \ell} \R^n)\), if \(df = 0\) and either \(\mathbf{t} f = 0\) or \(\mathbf{t} \varphi = 0\), then 
\[
  \Bigabs{\int_{\R^n} f \wedge \varphi}
  \le C \norm{f}_{\mathrm{L}^1} \bigl(\norm{d \varphi}_{\mathrm{L}^n} + \norm{\varphi}_{\mathrm{L}^n}\bigr).
\]
\end{theorem}

Theorem~\ref{theoremWeakDomain} is a consequence of the counterparts of theorem~\ref{theoremStrongEstimate} where the quantity \(\norm{df}_{\mathrm{W}^{-2,n / (n - 1)}}\) is replaced by a suitable dual quantity \cite{BrezisVanSchaftingen2007}*{lemmas 3.11 and 3.16}: 

\begin{theorem}
\label{theoremStrongDomain}
Let \(\Omega \subset \R^n\) be a smooth domain and let \(\ell \in \{1, \dotsc, n - 1\}\). There exists \(C > 0\) such that for every \(f \in C^\infty (\Bar{\Omega}; \bigwedge^{\ell} \R^n)\) and every \(\varphi \in C^\infty (\Bar{\Omega}; \bigwedge^{n - \ell} \R^n)\), then 
\begin{multline*}
  \Bigabs{\int_{\Omega} f \wedge \varphi }
  \le C \Bigl(\norm{f}_{\mathrm{L}^1} + \sup \Bigl\{\Bigabs{\int_{\Omega} f \wedge d\zeta} \st  \zeta \in C^\infty (\Bar{\Omega}) \\ \text{ and } \norm{D^2 \zeta}_{\mathrm{L}^n} \le 1 \Bigr\}\Bigr) \norm{D \varphi}_{\mathrm{L}^n}.
\end{multline*}
If moreover \(\mathbf{t} \varphi = 0\) on \(\partial \Omega\), then 
\begin{multline*}
  \Bigabs{\int_{\Omega} f \wedge \varphi }
  \le C \Bigl(\norm{f}_{\mathrm{L}^1} + \sup \Bigl\{\Bigabs{\int_{\Omega} f \wedge d\zeta} \st  \zeta \in C^\infty (\Bar{\Omega}), \mathbf{t} \zeta = 0 \text{ on \(\partial \Omega\)}\\\text{ and } \norm{D^2 \zeta}_{\mathrm{L}^n} \le 1 \Bigr\}\Bigr) \norm{D \varphi}_{\mathrm{L}^n}.
\end{multline*}
\end{theorem}

Theorem~\ref{theoremWeakDomain} can also be deduced from theorem~\ref{theoremWeakEstimate}: first theorem~\ref{theoremWeakDomain} is proved on a half-space by a reflection argument (see also \citelist{\cite{AmroucheNguyen2011a}\cite{AmroucheNguyen2011b}\cite{AmroucheNguyen2012}}), then it is extended to a general domain by local charts and partition of the unity \cite{BrezisVanSchaftingen2007}*{remark 3.3} (see also \cite{Xiang2013}).

\subsection{Other Sobolev spaces}
Besides the Sobolev space \(\dot{\mathrm{W}}^{1, n} (\R^n)\), there are other spaces that just miss the embedding in \(\mathrm{L}^\infty (\R^n)\): the Sobolev spaces \(\dot{\mathrm{W}}^{k, n/k} (\R^n)\) for \(k < n\), the Sobolev--Lorentz spaces \(\dot{\mathrm{W}}^{k, n/k, q} (\R^n)\), the fractional Sobolev--Slobodetski\u \i{} spaces \(\dot{\mathrm{W}}^{s, n/s} (\R^n)\), the Besov spaces \(\dot{\mathrm{B}}^{s, n/s}_q (\R^n)\) and the Triebel--Lizorkin spaces \(\dot{\mathrm{F}}^{s, n/s}_q (\R^n)\) for \(s > 1\) and \(q \ge 1\).

Let us first observe that in theorems~\ref{theoremWeakEstimate} and \ref{theoremCirculation}, the \(\dot{\mathrm{W}}^{1, n}\) norm can be replaced by any stronger norm; the same is true for theorem~\ref{theoremWeakHodge} with any space that is embedded in \(\mathrm{L}^n\).
The other cases are not covered straightforwardly. An inspection of the proof of theorem~\ref{theoremWeakEstimate} shows that the main ingredients are an embedding on hyperplanes into H\"older-continuous functions and a Fubini type theorem. The latter property for Triebel--Lizorkin spaces \citelist{\cite{Triebel1983}*{Theorem 2.5.13}\cite{Bourdaud}*{th\'eor\`eme 2}\cite{RunstSickel1996}*{theorem 2.3.4/2}} allows to extend  theorem~\ref{theoremWeakEstimate} \citelist{\cite{BourgainBrezis2004}*{remark 1}\cite{VanSchaftingen2004Divf}*{remark~5} \cite{BourgainBrezis2007}*{Remark 11}\cite{VanSchaftingen2010}*{proposition 2.1}}.

\begin{theorem}
Let \(\ell \in \{1, \dotsc, n - 1\}\), \(s > 0\) and \(q > 0\). There exists \(C > 0\) such that for every \(f \in C^\infty_c (\R^n; \bigwedge^{\ell} \R^n)\) and every \(\varphi \in C^\infty_c (\R^n; \bigwedge^{n - \ell} \R^n)\), if \(df = 0\), then 
\[
  \Bigabs{\int_{\R^n} f \wedge \varphi}
  \le C \norm{f}_{\mathrm{L}^1} \norm{\varphi}_{\dot{\mathrm{F}}^{s, n/s}_q}.
\]
\end{theorem}

The result can then be extended by classical embedding theorems to Sobolev--Lorentz spaces \(\mathrm{W}^{1, n, q} (\R^n)\) and Besov spaces \(\dot{\mathrm{B}}^{s, p}_q (\R^n)\) with \(q < \infty\) \citelist{\cite{VanSchaftingen2006BMO}*{remark~4.2}\cite{MitreaMitrea2009}\cite{VanSchaftingen2010}}. It is not known whether the results can be extended to the case \(q = \infty\). 

\begin{openproblem}[Critical estimate in Besov spaces \citelist{\cite{VanSchaftingen2010}*{open problem 1}\cite{VanSchaftingen2013}*{open problem 8.2}}]
Does there exist a constant \(C > 0\) such that for every \(f \in C^\infty_c (\R^n; \bigwedge^{\ell} \R^n)\) and every \(\varphi \in C^\infty_c (\R^n; \bigwedge^{n - \ell} \R^n)\), if \(df = 0\), then 
\[
  \Bigabs{\int_{\R^n} f \wedge \varphi}
  \le C \norm{f}_{\mathrm{L}^1} \norm{\varphi}_{\dot{\mathrm{B}}^{s, n/s}_\infty}?
\]
\end{openproblem}

\begin{openproblem}[Critical estimate in Sobolev--Lorentz spaces \citelist{\cite{BourgainBrezis2007}*{open problem 1}\cite{VanSchaftingen2010}*{open problem 2}\cite{VanSchaftingen2013}*{open problem 8.3}}]
Is there a constant \(C > 0\) such that for every \(f \in C^\infty_c (\R^n; \bigwedge^{\ell} \R^n)\) and every differential form \(\varphi \in C^\infty_c (\R^n; \bigwedge^{n - \ell} \R^n)\), if \(df = 0\), then 
\[
  \Bigabs{\int_{\R^n} f \wedge \varphi}
  \le C \norm{f}_{\mathrm{L}^1} \norm{D \varphi}_{\mathrm{L}^{n, \infty}}?
\]
\end{openproblem}

When \(\ell = 1\), by the embeddings of the Sobolev space \(\dot{\mathrm{W}}^{1,1} (\R^n)\) into the Besov space \(B^{s, n /(n + 1 - s)}_1 (\R^n)\) \cite{Kolyada}*{corollary 1} and into the Lorentz space \(\mathrm{L}^{\frac{n}{n - 1}, 1} (\R^n)\) \citelist{\cite{Alvino1977}\cite{Tartar1998}}, the inequalities holds \citelist{\cite{VanSchaftingen2010}}.

A positive answer to open problem 1 or open problem 2 would imply some limiting Sobolev type inequalities into \(\mathrm{L}^\infty\) which have been proved since \citelist{\cite{Mironescu2010}*{proposition 3}\cite{VanSchaftingen2013}*{p.~911}\cite{BousquetVanSchaftingen}}.

The extension of theorems~\ref{theoremStrongHodge} and \ref{theoremStrongEstimate} to the fractional case is more delicate. 
Theorem~\ref{theoremStrongHodge} has been extended when \(\ell = n - 1\) to some scale of Triebel--Lizorkin spaces \cite{BousquetMironescuRuss2013}.

\begin{theorem}
Let \(s \in (\frac{1}{2}, \frac{n}{2}]\) and \(q \in [2, n/s]\). If \(g \in F^{s-1,n/s}_q (\R^n; \bigwedge^{n} \R^n)\) and \(dg = 0\) in the sense of distributions, then there exists \(u \in \mathrm{L}^\infty (\R^n; \bigwedge^{n - 1} \R^n) \cap F^{s, n/s}_q (\R^n;\bigwedge^{n -1}\R^n) \), such that \(u\) is continuous and \(du = g\) in the sense of distributions.
Moreover,
\[
  \norm{u}_{\mathrm{L}^{\infty}} + \norm{u}_{F^{s, n /s}_q} \le C \norm{g}_{F^{s-1, n/s}_q}.
\] 
\end{theorem}

When \(s = \frac{n}{2}\) and \(q = 2\), the result goes back to V. Maz\cprime{}ya \cite{Mazya2007} (see also \citelist{\cite{Mironescu2010}\cite{MazyaShaposhnikova2009}}).

\subsection{Hardy inequalities}
Another question is whether the Sobolev inequality \eqref{ineqHodgeSobolev} has a corresponding Hardy inequality. A positive answer has been given by V. Maz\cprime{}ya \cite{Mazya2010} (see also \citelist{\cite{BousquetMironescu}\cite{BousquetVanSchaftingen}}):
\begin{equation}
\label{ineqHodgeHardy}
  \int_{\R^n} \frac{\abs{u (x)}}{\abs{x}}\dif x 
  \le C \int_{\R^n} \abs{d u} + \abs{d^* u}.
\end{equation}

H.\thinspace Castro, J.\thinspace  D{\'a}vila and Wang Hui have obtained another family of Hardy inequalities with cancellation phenomena \citelist{\cite{CastroWang2010}\cite{CastroDavilaWang2011}\cite{CastroDavilaWang2013}}:
\[
  \int_{\R^{n - 1} \times \R^+} \Bigabs{ D \Bigl( \frac{u (y)}{y_n}\Bigr)}\,dy
  \le C \int_{\R^{n - 1} \times \R^+} \abs{ D^2 u};
\]
their work is concerned in boundary singularities in the potential whereas we are concerned with point singularities.

\subsection{Larger classes of operators}
The Korn--Sobolev inequality of M.\thinspace{}J.\thinspace{} Strauss \cite{Strauss1973}
\begin{equation}
\label{ineqKornSobolev}
  \norm{u}_{\mathrm{L}^{n / (n - 1)}} \le C \norm{Du + (D u)^*}
\end{equation}
is a variant of the Gagliardo--Nirenberg--Sobolev inequality that plays a role in the study of maps of bounded deformation \citelist{\cite{AmbrosioCosciaDalMaso1997}\cite{TemamStrang1980}}. 
The components of the deformation tensor \(E u = (Du + (D u)^*)/2\) do not satisfy any first-order differential condition, so that theorem~\ref{theoremWeakEstimate} cannot be applied.
However, the tensor field \(Eu\) satisfies the Saint-Venant compatibility conditions:
\[
  \partial_{k} \partial_{l} (Eu)_{ij} + \partial_{i} \partial_{j} (Eu)_{kl}
  =\partial_{k} \partial_{j} (Eu)_{il} + \partial_{i} \partial_{l} (Eu)s_{kj}.
\]
In view of this H.\thinspace{}Brezis has suggested that theorem~\ref{theoremWeakEstimate} should hold when the derivative is replaced by higher-order conditions. This was proved for a class of 
second-order conditions \cite{VanSchaftingen2004ARB} (yielding an alternative proof to the Korn--Sobolev inequality \eqref{ineqKornSobolev} \citelist{\cite{VanSchaftingen2004ARB}\cite{BourgainBrezis2007}*{corollary 26}}) and for a class of higher-order conditions \cite{BourgainBrezis2007}*{corollary 14}.
The differential operators for which theorems~\ref{theoremWeakEstimate} and \ref{theoremStrongEstimate} hold have been characterized \cite{VanSchaftingen2013}*{theorem 1.4, proposition 2.1 and theorem 9.2}.

\begin{theorem}
\label{theoremCocanceling}
Let \(L(D)\) be a homogeneous differential operator on \(\R^n\) from \(E\) to \(F\).
The following conditions are equivalent
\begin{enumerate}[(i)]
\item 
\label{itemCocancelingEstimate} there exists \(C > 0\) such that for every \(f \in \mathrm{L}^1(\R^n; E)\) such that \(L(D)f=0\) and \(\varphi \in C^\infty_c(\R^n; E)\),
\[
  \Bigabs{\int_{\R^n} f \cdot \varphi} \le C \norm{f}_{\mathrm{L}^1} \norm{D \varphi}_{\mathrm{L}^n},
\]
\item 
\label{itemBourgainBrezis} for every \(f \in \mathrm{L}^1(\R^n; E)\), one has 
\(f \in\dot{\mathrm{W}}^{-1,\frac{n}{n - 1}} (\R^n; E)\) if and only if \(L(D)f\in \dot{\mathrm{W}}^{-1-k,\frac{n}{n - 1}} (\R^n; F);
\)
moreover, 
\[
  \norm{f}_{\dot{\mathrm{W}}^{-1, n/(n - 1)}} \le C \bigl(\norm{f}_{\mathrm{L}^1} + \norm{L(D)f}_{\dot{\mathrm{W}}^{-1-k, n/(n - 1)}}\bigr),
\]
\item 
\label{itemCocanceling0} for every \(f \in \mathrm{L}^1(\R^n; E)\) such that \(L(D)f=0\)
\[
  \int_{\R^n} f = 0,
\]
\item \label{itVanishDirac} for every \(e \in E\), if \(L(D)\, (\delta_0 e) =0\), then \(e = 0\),
\item \label{itemCocancelingCocanceling} \(L(D)\) is cocanceling.
\end{enumerate}
\end{theorem}

The \emph{cocancellation} condition is a new condition that was introduced in order to solve this problem \cite{VanSchaftingen2013}*{definition 1.3}.

\begin{definition}
\label{definitionCocanceling}
Let \(L(D)\) be a homogeneous linear differential operator on \(\R^n\) from \(E\) to \(F\). The operator \(L(D)\) is \emph{cocanceling} if 
\[
  \bigcap_{\xi \in \R^n \setminus \{0\}} \ker L(\xi)=\{0\}.  
\]
\end{definition}

Theorem~\ref{theoremCocanceling} implies that the elements of the kernel of a differential operator acting  vector measures from on \(\R^n\)
define linear functionals on \(\dot{\mathrm{W}}^{1, n} (\R^n)\) if and only if the kernel does not contain any Dirac measure. In particular, if such a measure does not charge points, it does not charge sets of null \(\mathrm{W}^{1, n}\)--capacity.

The main analytical difficulty in theorem~\ref{theoremCocanceling} is proving \eqref{itemBourgainBrezis}.  When \(E = \R\) it can be proved by theorem~\ref{theoremMainClosed}; as the latter deals naturally only with first-order differential operators, this requires decomposing in a suitable fashion higher-order differential operators \cite{VanSchaftingen2008}*{proof of theorem 8}.
The weaker \eqref{itemCocancelingEstimate} can be obtained directly when \(E = \R\) following the lines of the proof of theorem~\ref{theoremWeakEstimate} \cite{VanSchaftingen2008}*{proof of theorem 5}; this argument has been adapted to the fractional case \cite{VanSchaftingen2008}*{(4)}.
The passage to the vector case is done by an algebraic construction
\cite{VanSchaftingen2013}*{lemma~2.5}.

The inequality \eqref{itemCocancelingEstimate} appears with \(L (D) = \Div^2\) in the homogenization of stiff heterogeneous plates \cite{BrianeCasadoDiaz2012}*{lemma~15}.

An interesting consequence is the characterization of the operators such that a Gagliardo--Nirenberg--Sobolev inequality or a Hardy inequality hold \citelist{\cite{VanSchaftingen2013}*{theorem 1.3 and proposition 6.1}\cite{BousquetVanSchaftingen}}.

\begin{theorem}
\label{theoremCanceling}
Let \(A(D)\) be an elliptic homogeneous linear differential operator of order
\(k\) on \(\R^n\) from \(V\) to \(E\). The following conditions are equivalent
\begin{enumerate}[(i)]
\item \label{itemCancelingSobolev} for every \(u \in C^\infty_c(\R^n; V)\),
\[
  \norm{D^{k-1}u }_{\mathrm{L}^{\frac{n}{n - 1}}} \le C\norm{A(D) u}_{\mathrm{L}^1},
\]
\item \label{itemCancelingHardy} for every \(u \in C^\infty_c(\R^n; V)\),
\[
  \int_{\R^n} \frac{\abs{D^{k-1}u (x)}}{\abs{x}} \dif x \le C\norm{A(D) u}_{\mathrm{L}^1},
\]
\item\label{itCancelingVanishCinfty} for every \(u \in C^\infty(\R^n; V)\), if 
\(\supp A(D)u\) is compact, then 
\[
  \int_{\R^n} A(D) u = 0,
\]
\item \(A(D)\) is canceling.
\end{enumerate}
\end{theorem}

In \eqref{itCancelingVanishCinfty} it is essential to consider vector fields \(u\) that do not have compact support and do not tend to \(0\) too fast at infinity.
The ellipticity condition is the classical notion of ellipticity for overdetermined differential operators \citelist{\cite{Hormander1958}*{theorem 1}\cite{Spencer}*{definition 1.7.1}} (when \(\dim V = 1\), see also S.\thinspace Agmon \citelist{\cite{Agmon1959}*{\S 7}\cite{Agmon1965}*{definition 6.3}}).

\begin{definition}
A homogeneous linear differential operator \(A(D)\) on \(\R^n\) from \(V\) to \(E\) is \emph{elliptic} if for every \(\xi \in \R^n \setminus\{0\}\), 
\(A(\xi)\) is one-to-one.
\end{definition}

The cancellation is a new condition dual to cocancellation which was introduced to characterize such operators \cite{VanSchaftingen2013}*{definition 1.2}.

\begin{definition}
A homogeneous linear differential operator \(A(D)\) on \(\R^n\) from \(V\) to \(E\) is \emph{canceling} if 
\[
  \bigcap_{\xi \in \R^n \setminus\{0\}} A(\xi)[V]=\{0\}.
\]
\end{definition}

Theorem~\ref{theoremCanceling} gives in particular the Hodge--Sobolev inequality \eqref{ineqKornSobolev} and the Korn--Sobolev inequality \eqref{ineqHodgeSobolev} as the corresponding differential operators are canceling \cite{VanSchaftingen2013}*{propositions 6.4 and 6.6}. Further examples of higher-order canceling operators which are a higher-order analogue of the Hodge complex were recently given \cite{LanzaniRaich2014}.

The proof of \eqref{itemCancelingSobolev} and \eqref{itemCancelingHardy} are based on the construction of a differential operator such that for every \(\xi \in \R^n \setminus \{0\}\), \(\ker L (\xi) = A (\xi)[V]\) \cite{VanSchaftingen2013}*{proposition 4.2}.
The latter is a combination of a classical commutative algebra result of L.\thinspace Ehrenpreis \citelist{\cite{Ehrenpreis}\cite{Komatsu}*{theorem 2}\cite{Spencer}*{theorem 1.5.5}} which states that every submodule of the module of differential operators is finitely generated, and the application of ellipticity.

These estimate generalize to fractional Sobolev spaces and Sobolev Lorentz space as the Hodge estimates above. In particular, since the derivative on \(\R^n\) is canceling if and only \(n \ge 2\), we recover that the inequality \citelist{\cite{BourgainBrezisMironescu2004}*{lemma~D.1}\cite{SchmittWinkler}*{proposition 4}\cite{Solonnikov1975}*{theorem 2}\cite{Kolyada}*{theorem 4}\cite{CDDD}*{theorem 1.4}\cite{VanSchaftingen2013}*{corollary 8.2}}
\[
 \norm{u}_{\dot{\mathrm{W}}^{s, n/(n - (1-s))}} \le C \norm{D u}_{\mathrm{L}^1}
\]
holds for every \(u \in C^\infty_c(\R^n)\) if and only if \(n \ge 2\).

If \(A (D)\) is elliptic and canceling, then \cite{BousquetVanSchaftingen}
\[
 \norm{D^{k-n} u}_{\mathrm{L}^\infty} \le C \norm{A (D) u}_{\mathrm{L}^1}.
\]
The cancellation condition is not necessary \citelist{\cite{VanSchaftingen2013}*{remark 5.1}\cite{BousquetVanSchaftingen}}. Indeed, the derivative \(D\) on \(\R\) is not canceling but the inequality \(\norm{u}_{\mathrm{L}^\infty} \le \norm{u'}\) holds nevertheless.

The ellipticity assumption in theorem~\ref{theoremCanceling} is necessary for the first-order Sobolev inequality of \eqref{itemCancelingSobolev} \cite{VanSchaftingen2013}*{theorem 1.3 and corollary 5.2} and for Hardy--Sobolev inequalities \cite{BousquetVanSchaftingen}; it is not necessary for Hardy inequalities of the form \eqref{itemCancelingHardy} \cite{BousquetVanSchaftingen} nor for higher-order Sobolev inequalities \cite{VanSchaftingen2013}*{proposition 5.4}.

For the Hodge operator, \(A (D)=(\delta, d)\), theorem~\ref{theoremStrongDomain} implies that if either the tangential or the normal component vanishes on the boundary (either \(\mathbf{t} u = 0\) or \(\mathbf{n} u = 0\) on \(\partial \R^n_+\)), then 
\[
 \norm{u}_{\mathrm{L}^{n/(n - 1)}} \le C \bigl(\norm{du} + \norm{d^* u}\bigr).
\]
It would be interesting to define a class of canceling boundary conditions that ensure the validity of Sobolev estimates on half-spaces.

\begin{openproblem}
If \(A (D)\) is an elliptic canceling operator, under what boundary conditions on \(\partial \R^n_+ = \R^{n - 1} \times \{0\}\) does the inequality
\[
 \norm{D^{k-1}u }_{\mathrm{L}^{\frac{n}{n - 1}} (\R^n_+)} \le C\norm{A(D) u}_{\mathrm{L}^1 (\R^n_+)}
\]
hold?
\end{openproblem}

It would be natural to investigate boundary conditions that satisfy the Lopatinski\u \i--Shapiro ellipticity condition (also known as coerciveness or covering conditions) \citelist{\cite{Hormander1985}*{definition 20.1.1}\cite{Lopatinskii}\cite{Solonnikov1971}}.

The reader will have noted that we did not state in this section any counterpart of theorem~\ref{theoremWeakHodge} and \ref{theoremStrongHodge}.

\begin{openproblem}
State necessary and sufficient conditions on \(K (D)\) such that for every \(v \in \dot{\mathrm{W}}^{1, n} (\R^n; E)\), there exists \(u \in \dot{\mathrm{W}}^{1, n} (\R^n; E)\) such that \(K (D) u = K (D) v\) and 
\[
  \norm{D u}_{\mathrm{L}^n} + \norm{u}_{\mathrm{L}^\infty} \le C \norm{D v}_{\mathrm{L}^n}.
\]
\end{openproblem}

\subsection{Noncommutative situations}
A nilpotent homogeneous group \(G\) is a connected and simply connected Lie group whose Lie algebra \(\mathfrak{g}\) of left-invariant vector fields is graded, nilpotent and stratified:
\begin{enumerate}[(a)]
\item \(\mathfrak{g}=\mathfrak{g}_1 \oplus \mathfrak{g}_2 \oplus \dots \oplus \mathfrak{g}_p\),
\item \([\mathfrak{g}_i,\mathfrak{g}_j] \subset \mathfrak{g}_{i+j}\) for \(i+j \le p\) and \([V_i,V_j]=\{0\}\) if \(i+j > p\), 
\item \(\mathfrak{g}_1\) generates \(\mathfrak{g}\) by Lie brackets.
\end{enumerate}
These spaces are a good framework for defining homogeneous Sobolev spaces \cite{FollandStein1974}, Hardy spaces \cite{FollandStein1982} and studying singular integrals \cite{Stein1993}*{theorem XII.4}.
The homogeneous dimension \(Q=\sum_{j=1}^p j m_j\) plays an important role in properties of \(G\). 
In particular, the following Sobolev embedding holds for \(p < Q\) for the nonisotropic Sobolev space \(\mathrm{N\dot{W}}^{1, p} (G)\) \cites{Folland1975,FollandStein1974,RothschildStein1976}:
\[
  \mathrm{N\dot{W}}^{1,p} (G)=\{ u \in \mathrm{L}^p(G) \st D_b u \in \mathrm{L}^p(G)\}  \subset \mathrm{L}^{\frac{Qp}{Q-p}} (G)\;,
\]
where the horizontal derivative \(D_b u\) is the pointwise restriction of \(Du\) to the horizontal bundle \(T_b G = \mathfrak{g}_1\). 
As these embeddings are counterparts of the classical Sobolev embedding in the Euclidean space, theorem~\ref{theoremWeakEstimate} has a counterpart on homogeneous groups \cite{ChanilloVanSchaftingen2009}*{theorem 1}.

\begin{theorem}
\label{theoremGroups}
If \(f \in C^\infty_c (G; T_bG)\) is a section and for every \(\psi \in C^\infty_c(G)\),
\[
 \int_{G} \dualprod{D_b \varphi}{f} = 0,
\]
and if \(\varphi \in C^\infty_c(G;T^*_bG)\) is a section, then
\[
  \Bigl\lvert\int_{G} \dualprod{\varphi}{f}  \Bigr\rvert\le C \norm{f}_{\mathrm{L}^1(G)}\norm{\nabla_b \varphi}_{\mathrm{L}^Q(G)}\;.
\]
\end{theorem}

Theorem~\ref{theoremGroups} is proved following the strategy of the proof of theorem~\ref{theoremWeakEstimate}; the main point is to split with Jerison's machinery for analysis on nilpotent homogeneous groups \cite{Jerison1986} a function on a normal subgroup into a function which is controlled in \(\mathrm{L}^\infty\) and another which is the restriction of function with a horizontal derivative controlled in \(\mathrm{L}^\infty\) \cite{ChanilloVanSchaftingen2009}*{lemma~2.1}.
The proof also gives fractional estimates \cite{ChanilloVanSchaftingen2009}*{theorem 4}.
Theorem~\ref{theoremGroups} gives Gagliardo--Nirenberg--Sobolev inequalities for \((0, q)\) forms in the \(\Bar{\partial}_b\) complex of classes of CR manifolds \cite{Yung2010}*{theorems 2 and 3} and for involutive structures \cite{HouniePicon2011}. 

Following the ideas of \cite{VanSchaftingen2008}, theorem~\ref{theoremGroups} has a higher--order analogue in which \(f\) is a \emph{symmetric} tensor-field and the condition is replaced by a higher order condition \cite{ChanilloVanSchaftingen2009}*{theorem 5 and lemma~5.3}. 

\begin{theorem}
\label{theoremHigherOrderGroups}
If \(f \in C^\infty_c (G; \otimes^k T_bG)\) is a section and for every \(\psi \in C^\infty_c(G)\),
\[
 \int_{G} \dualprod{D_b^k \varphi}{f} = 0,
\]
and if \(\varphi \in C^\infty_c(G;\operatorname{Sym}^k (T_bG))\) is a section, then
\[
  \Bigl\lvert\int_{G} \dualprod{\varphi}{f}  \Bigr\rvert\le C \norm{f}_{\mathrm{L}^1(G)}\norm{\nabla_b \varphi}_{\mathrm{L}^Q(G)}\;.
\]
\end{theorem}

Here \(\operatorname{Sym}^k (T_bG)\) denotes the bundle of symmetric \(k\)-linear maps on the horizontal bundle. It is not known whether the symmetry assumption is necessary for this inequality to hold \cite{ChanilloVanSchaftingen2009}*{open problem 1}.
This result was used to obtain the Gagliardo--Nirenberg inequality for forms in the Rumin complex \citelist{\cite{Rumin1994}\cite{Rumin1999}\cite{Rumin2000}\cite{Rumin2001}} for forms on the Heisenberg groups \(\mathbb{H}^1\) and \(\mathbb{H}^2\) \cite{BaldiFranchi}. Since the Rumin complex contains higher-order differential operators, the higher-order estimates play a crucial role in the proof. The method would require heavy explicit computations of complexes to be generalized to higher-order Heisenberg groups. 

The counterparts of the stronger theorems~\ref{theoremStrongHodge} and \ref{theoremStrongEstimate} have been obtained by Yi Wang and Po-Lam Yung \cite{WangYung} following the strategy of Bourgain and Brezis \cite{BourgainBrezis2007}.

The study of canceling and cocanceling operators on noncommutative homogeneous groups remains widely open.

\section{Coda: characterizing functions satisfying the estimates}

\subsection{The relationship with bounded mean oscillation}
The results above can all be thought as substitutes for the failing embedding of \(\dot{\mathrm{W}}^{1, n} (\R^n)\) into \(\mathrm{L}^\infty (\R^n)\) when \(n \ge 2\). 
A classical substitute for this embedding is the embedding of \(\dot{\mathrm{W}}^{1, n} (\R^n)\) into the space of functions of bounded mean oscillations \(\mathrm{BMO} (\R^n)\) \cite{BrezisNirenberg1995}*{example 1}.

H.\thinspace{}Brezis has suggested investigating the relationship between this embedding and theorem~\ref{theoremWeakEstimate} by studying  for \(\ell \in \{1, \dotsc, n - 1\}\) the space \(D_\ell (\R^{n})\) of measurable functions \(\varphi : \R^n \to \R\) for which the semi-norm
\[
\begin{split}
 \abs{\varphi}_{\mathrm{D}_{\ell}}
 =\sup \Bigl\{\Bigabs{\int_{\R^n} \varphi f \wedge e} : f \in C^\infty_c \bigl(\R^n; \textstyle{\bigwedge^\ell} \R^n\bigr),\; &e \in \textstyle\bigwedge^{n - \ell} \R^n \\
 &\text{ and } \int_{\R^n} \abs{e}\abs{f} \le 1\Bigr\}
 \end{split}
\]
is finite \cite{VanSchaftingen2006BMO}*{definition 2.4}.
Theorem~\ref{theoremWeakEstimate} is equivalent with the embedding
\[
 \dot{\mathrm{W}}^{1, n} (\R^n) \subset \mathrm{D}_{n - 1} (\R^n).
\]
By an observation of F. Bethuel, G. Orlandi and D. Smets, this embedding is strict \cite{BethuelOrlandiSmets2004}*{remark 5.4}.

These new spaces form a monotone family between the classical spaces \(\mathrm{L}^\infty (\R^n)\) and \(\mathrm{BMO} (\R^n)\) \cite{VanSchaftingen2006BMO}*{proposition 2.9, theorem 3.1, proposition 4.6, proposition 5.1, theorem 5.3},
\[
 \mathrm{L}^\infty (\R^n) \subsetneq \mathrm{D}_{n - 1} (\R^n) \subsetneq \dotsb \subsetneq \mathrm{D}_1 (\R^n) \subsetneq \mathrm{BMO} (\R^n).
\]
The embeddings are strict since 
\[
  \log \Bigl(\sum_{i = 1}^\ell \abs{x_i}^2 \Bigr) \in \mathrm{D}_k (\R^n)
\]
if and only if \(\ell \le n - k\) \cite{VanSchaftingen2006BMO}*{proposition 4.6}.
It can be also noticed that \(\mathrm{VMO} (\R^n)\) is not a subset of \(D_1 (\R^n)\) \cite{VanSchaftingen2006BMO}*{proposition 5.1}.

The conclusion of this study is that the estimates of theorem~\ref{theoremWeakEstimate} are stronger than the embedding of \(\dot{\mathrm{W}}^{1, n} (\R^n)\) into \(\mathrm{BMO} (\R^n)\).

\subsection{Strong charges}

It is possible to characterize the distributions \(F\) such that there exists \(f \in C (\R^n; \bigwedge^{n - 1} \R^n)\) such that \(df = F\) \citelist{\cite{DePauwPfeffer2008}*{theorem 4.8}\cite{Pfeffer2012}*{chapter 11}}.

\begin{theorem}
\label{theoremStrongCharges}
Let \(F : \R^n \to \bigwedge^n \R^n\) be a distribution. There exists \(f \in C (\R^n; \bigwedge^{n - 1} \R^n)\) such that \(F = df\) if and only if \(F\) is a strong charge.
\end{theorem}

Strong charges were introduced to solve this problem \citelist{\cite{DePauwPfeffer2008}\cite{Pfeffer2012}*{chapter 10}}.

\begin{definition}
The distribution \(F : \R^n \to \bigwedge^n \R^n\) is a \emph{strong charge} if for every \(\varepsilon > 0\) and every \(R > 0\), there exists \(C > 0\) such that if \(\varphi \in C^\infty_c (B_R)\),
\[
  \Bigabs{\int_{\R^n} F \wedge \varphi} \le C \norm{\varphi}_{L^1} + \varepsilon \norm{d \varphi}_{L^1}.
\]
\end{definition}

In particular, functions in \(L^n (\R^n)\) define strong charges \cite{DePauwPfeffer2008}*{proposition 2.9}. Theorem~\ref{theoremStrongCharges} provides thus an alternate proof of theorem~\ref{theoremWeakHodge}.
The case of \(\ell\) forms with \(\ell \in \{1, \dotsc, n - 2\}\) has also been studied \cite{DePauwMoonensPfeffer2009}.

\end{document}